\documentclass[10pt,amsfonts, epsfig]{amsart}
\usepackage{amsmath, amscd, amssymb}
\usepackage{graphpap, color}
\usepackage{mathrsfs}
\usepackage{pstricks}
\usepackage{color}
\usepackage{cancel}
\usepackage[mathscr]{eucal}
\usepackage{pstricks}
\usepackage{color}
\usepackage{cancel}
\usepackage{verbatim}
\usepackage{latexsym}
\usepackage[all]{xy}

\usepackage{lipsum}
\makeatletter
\g@addto@macro{\endabstract}{\@setabstract}
\newcommand{\authorfootnotes}{\renewcommand\thefootnote{\@fnsymbol\c@footnote}}%
\makeatother

\def\bB{{\mathbf B}}

\def\cA{{\cal A}}

\numberwithin{equation}{section}

\def\sO{{\mathscr O}}
\def\sM{{\mathscr M}}

\def\sV{{\mathscr V}}

\def\sO{\mathscr{O}}

\def\sV{\mathscr{V}}

\newcommand{\CC}{\mathbb{C}}

\newcommand{\PP}{\mathbb{P}}

\newcommand{\ZZ}{\mathbb{Z}}
\newcommand{\NN}{\mathbb{N}}

\def\sub{{\subset}}

\newcommand{\bw}{\mathbf{w}}

\newcommand{\cal}{\mathcal}

\def\cT{{\cal T}}

\def\mapright#1{\,\smash{\mathop{\lra}\limits^{#1}}\,}

\def\sta{^\ast}

\def\sta{^{\ast}}

\def\sta{^*}

\def\lra{\longrightarrow}

\newcommand{\ep}{\epsilon}
\newcommand{\lam}{\lambda}

\newcommand{\coh}{H}

\def\begeq{\begin{equation}}
\def\endeq{\end{equation}}
\def\and{\quad{\rm and}\quad}
\def\bl{\bigl(}
\def\br{\bigr)}

\def\sub{\subset}

\def\Po{{\mathbb P^1}}
\def\and{\quad\text{and}\quad}

\DeclareMathOperator{\tr}{tr} 
 \DeclareMathOperator{\rank}{rank}

\DeclareMathOperator{\Res}{Res}

\newtheorem{prop}{Proposition}[section]

\newtheorem{theo}[prop]{Theorem}
\newtheorem{lemm}[prop]{Lemma}
\newtheorem{coro}[prop]{Corollary}
\newtheorem{rema}[prop]{Remark}
\newtheorem{exam}[prop]{Example}
\newtheorem{defi}[prop]{Definition}

\newtheorem{assu}[prop]{Assumption}
\newtheorem{setup}[prop]{Setup}

\newtheorem{defi-prop}[prop]{Definition-Proposition}

\def\dbar{\overline{\partial}}

\def\sta{^\ast}

\def\sO{{\mathscr O}}

\def\beq{\begin{equation}}
\def\eeq{\end{equation}}

\def\vsp{\vskip5pt}
\def\Pf{{\PP^4}}

\def\bee{\begin{equation}}
\def\eeq{\end{equation}}

\def\ti{\tilde}

\date{}

\title{Virtual Residue and an integral formalism }

\date{}

\author{Huai-Liang Chang$^*$  }
\address{Department of Mathematics, Hong Kong University of Science and Technology, Hong Kong} \email{mahlchang@ust.hk}
\thanks{${}^*$Partially supported by  Hong Kong GRF grant 16301515  and 16301717}

\author{Mu-Lin Li}
\address{College of Mathematics and Econometrics, Hunan University, China} \email{mulin@hnu.edu.cn}

\begin{document}
\maketitle

\begin{abstract} 	
	We generalize Grothendieck's residues $Res\frac{\psi}{s}$ to virtual cases, namely cases when the zero loci of the section $s$ has dimension larger than the expected dimension(zero). We also provide an exponential type integral formalism for the virtual residue, which can be viewed as an analogue of the Mathai-Quillen formalism for localized Euler classes.
 \end{abstract}

\black

 \section{Introduction}

 During the '90s, physicists studied ``genus zero B-twisted superconformal N=2 model" for the Landau-Ginzburg\footnote{A Landau Ginzburg space is a pair $(X,W)$ of a complex manifold $X$ and a holomorphic function $W:X\to\CC$ with compact critical locus;} space $(\CC^n,\bw)$ where $\bw$ is a holomorphic function on $\CC^n$.  Their ``physical states" lie in a ``chiral primary ring" $$\mathcal{R}=\mathbb{C}[x_1,\cdots,x_n]/(\partial_1\bw,\cdots,\partial_n \bw),$$ also known as the Milnor ring.  The ``correlation" between given states $\{F_i\}$'s were calculated by Vafa in \cite{Va}, using path integral argument,
 \beq\label{correlator1}
\langle\mathcal{O}_1(x)\cdots \mathcal{O}_k(x)\rangle= \sum_p\frac{F_1(p)\cdots F_k(p)}{H(p)},
\eeq
where $\mathcal{O}_i(x)\in \mathcal{R}$ is the $F_i$, $H=det(\partial_i\partial_j \bw)$ is the Hessian of $\bw$, and $p$ runs through (finitely many) closed points of the critical locus of $\bw$, which is assumed  to be zero dimensional and nonreduced. \\

 The form \eqref{correlator1} is then known to be identical to  the Grothendieck Residue
 \beq\label{correlator2}\sum_{p:d\bw(p)=0}Res_{p\in\CC^n}\frac{F_1\cdots F_k}{(\partial_1\bw,\cdots,\partial_n\bw)}.\eeq
 Such an interpretation generalizes \eqref{correlator1} to the case critical locus of $\bw$ is nonreduced (still zero dimensional). For example, in the case $n=5$ and $ \bw=x_1^5+\cdots+x_5^5$ 
  the critical locus $(d\bw=0)=(x_1^4=x_2^4=\cdots=x_5^4=0)$ is a nonreduced \textit{zero} dimensional scheme. In this case
 the Grothendieck Residue \eqref{correlator2} is understood as the  correlator for a ``genus zero B-twisted LG model $(\CC^5,\bw=\sum x_i^5)$". \\

  A complete (coordinate-free) definition of the Grothendieck residue is provided in the book of Griffiths and Harris  \cite[Chapter 5]{GH}, in which the zero locus of $d\bw$ needs to be assumed zero dimensional. However one may encounter the following case. Say $K_\Pf$ is the total space of the canonical line bundle $\sO_\Pf(-5)$ of $\Pf$ and let $W:K_\Pf\to \CC$ be defined by the pairing $\sO_\Pf(-5)\otimes \sO_\Pf(5)\to \sO_\Pf$
 with the quintic section $\sum_{i=1}^5 x_i^5$ of $\sO_\Pf(5)$. Usually the expression $W=p\sum_{i=1}^5 x_i^5$ is used where $p$ stands for the local coordinate on $K_\Pf$ in the noncompact direction. \\

  Let $\Pf\sub K_\Pf$ be the zero section and let $Q$ be the quintic hypersurface $(x_1^5+\cdots+x_5^5=0)\sub \Pf$. Then the form $W=p\sum x_i^5$ implies the critical loci $(dW=0)$ is identical to $Q$ as subschemes of $K_\Pf$, which is not zero dimensional and a ``Grothendieck Residue" $$Res_{Q\sub K_\Pf}\frac{\cdot}{dW},$$  expected to be responsible for the ``genus zero  B-twisted   correlator of the   Landau-Ginzburg space $(K_\Pf,W)$", is not
  defined to the best of authors' knowledge\footnote{While certain Hodge theoretical properties of  $(K_\Pf,W)$ were discussed in \cite{Fan};}. On the contrary, in A side, a sequence of works \cite{CL1, CL2, CLL, CLLL1, CLLL2, Ch, CR, FJR3, FJR2, FJR1, JKV} study mathematically all genus constructions and properties of  LG spaces including $(K_\Pf,W)$ and $([\CC^5/\ZZ_5],\bw)$. \\

      In this paper we generalize the Grothendieck Residue as follows. Let $M$ be a complex manifold (usually noncompact) and  let ${V}$ be a holomorphic bundle over $M$ with $\rank V=\dim M=n$. Let $s$ be a  holomorphic section of  $V$ with compact zero loci $Z$. Given any ``weight"   $$\psi\in \Gamma(M,K_M\otimes \det{V})$$
the Koszul complex   of $(V,s)$ associates a closed form $\eta_\psi\in \Omega^{n,n-1}(M-Z)$ via Griffiths-Harris's construction (\cite[Chapter 5]{GH}, also \eqref{ONE}).\\

  We then define the residue as
\beq\label{V-residue}Res_{Z}\frac{\psi}{s}:=\frac{1}{(2\pi i)^n}\int_{N}\eta_\psi\in\CC\eeq
where   $N$ is a real $2n-1$ dimensional piecewise smooth compact subset of $M$ that ``surrounds $Z$", in the sense that  $N=\partial T$ for some compact domain $T\sub M$, which contains $Z$ and is homotopically equivalent to $Z$. In Section 2 we show that such $N$ always exists and \eqref{V-residue} is independent of the choice of $N$. \black\\

 The residue thus defined is named ``Virtual Reside" when $\dim Z>0$ (note that $0=\dim M-\rank V$ is the expected dimension of the Kuranishi model $(M,V,s)$). Therefore it generalizes the classical zero dimensional case of Grothendieck Residues \cite[Chapter 5]{GH}.   
 \\ \black

  The main part of this paper is to provide an exponential type integral form of the defined virtual residue.

\begin{theo} \label{th1}
   Pick a Hermitian metric $h$ on $V$ and let $\nabla$ be its associated Hermitian connection with $\nabla^{0,1}=\dbar$. Let $\xi=-(*,s)_h$ be a smooth section of $V^*$ and $$S=-|s|^2+\dbar\xi\in\oplus_{p=0,1} \Omega^{(0,p)}(\wedge^p V^*).$$
   Assuming polynomial growth conditions for $s$ and $\nabla s$ (Assumption \ref{cond}), one has
 \beq\label{diff}\Res_Z\frac{\psi}{s}=\frac{(-1)^n}{(2\pi i)^n}\int_M(\psi\lrcorner e^{S}).\eeq
 Here $\lrcorner$ is the operation  contracting $\det V$ with $\det V^*$ so that $\psi\lrcorner e^{S}\in\Omega^{\ast,\ast}_M$.
 \end{theo}

 The formula may be viewed as an analogue of the Mathai-Quillen's integral formalism \cite{MQ}  of the localized Euler class $e_{s,loc}(V)$ for the Kuranishi model $(M,V,s)$.\\

 For LG space $(X,W)$, one simply takes $M=X,V=\Omega_X,s=dW$, and the weight $\psi$ is given as physics  ``observable",
 the formula's right hand side is usually referred as the ``zero mode of the genus zero B-twisted path integral\footnote{Contribution of constant maps to some integral defined over the space of all smooth degree zero maps from $\Po$ to $M$}".  It  means that virtual residue is an algebro-geometric model of the ``genus zero  B-twisted correlator". 
 \black
 \\

 The virtual residue in the case $V=T_M$ (and thus $s$ is a holomorphic vector field) should be related to the holomorphic equivariant cohomology (for example \cite{Liu} and \cite{Bru}). We also expect some purely algebraic construction can be used  to construct virtual residues in  arbitrary characteristic. Whether if there exists a more general algebro-geometric model governing higher genus B-twisted theory for any LG space $(X,W)$, (could it be related to a virtual residue over some moduli containing $M_{g,n}$, with the weight provided by conformal blocks?) is an interesting question. \\ 

{\bf Convention}. In this paper, for a  holomorphic bundle $V$ over a complex manifold $M$, we use $\Gamma(M,V)$ to denote the space of global holomorphic sections of the bundle $V$. Also all tensor products of bundles are over $\CC$. \\
\black

\vsp
\noindent{\bf Acknowledgment}:   The authors thank  Ugo Bruzzo, Jun Li, Eric Sharpe, Si Li, Qile Chen, Zheng Hua,  Huijun Fan, Yongbin Ruan, Edward Witten for helpful discussions. Special thanks to Si Li for informing us the operators $T_\rho,R_\rho$ in section three. Finally, we would like to
express our appreciation to the referee for pointing out how to improve the paper and providing many valuable suggestions in rewriting the manuscript to make the paper more readable. 

 \section{Construction of virtual residue}
 \subsection{Classical Grothendieck residue}
 We review the classical setup of Grothendieck residues. Let $B$ be the ball $\{z\in \mathbb{C}^n:|z|<\epsilon\}$ and $f_1,\cdots, f_n\in \mathcal{O}(\bar{B})$ functions holomorphic in a neighborhood of the closure. Suppose a ``transverse condition" is satisfied.

\begin{assu}\label{transverse0}
 The only common zero of $f_1,\cdots,f_n$ is the origin.
 \end{assu}

  Denote $\omega$ to be a meromorphic form as following
$$
\omega=\frac{g(z)dz_1\wedge\cdots \wedge dz_n}{f_1\cdots f_n}  \ \ \ \ \ (g\in\mathcal{O}(\bar{B})).
$$
 Pick a positive $\delta<<\ep$ and let $\Gamma$ be the real $n$-cycle defined by
$$
\Gamma=\{z:|f_i(z)|=\delta\},
$$
with the orientation given by $d(\mbox{arg} f_1)\wedge\cdots\wedge d(\mbox{arg} f_n)\ge0$. Then the residue is defined as
\beq\label{clas}
\Res_{0\in B}\omega=\big(\frac{1}{2\pi i}\big)^n\int_{\Gamma}\omega.
\eeq

 This number is independent of the coordinate. In \cite[Lemma in Page 651]{GH}
the number \eqref{clas} is also identified as $\int_{S^{2n-1}}\eta_{\omega}$
where $S^{2n-1}$ is
a real $2n-1$-dimensional sphere  centered at origin and contained in $B$, and $\eta_{\omega}$ is some closed $(n,n-1)$ form over $B$  constructed by a Koszul complex that we shall use in next subsection.

  One can view the Grothendieck residue \eqref{clas}   to be associated to the complex manifold $B$,   a section $(f_1,\cdots,f_n)$ of the trivial bundle $B\times \CC^n$ over $B$ whose (reduced) zero loci is $0\in B$, and a ``weight" $$g(z)dz_1\wedge\cdots\wedge dz_n\in \Gamma(B, K_B\otimes \det V).$$ 
 More generally one can consider the following.

{\bf Griffiths-Harris Set-up}\label{GH-Setup}.\\
{\sl
(1). ${V}$ is a holomorphic bundle over a complex manifold $M$,
$\dim M=\rank{V}=n;$\\
(2). $s\in\Gamma(M,V) $; the zero loci of $s$ is a compact set  $Z\sub  M$; $\dim Z=0$;\\
(3). $\psi\in \Gamma(M,K_M\otimes \det{V})$, called ``weight".}   \vsp



Nearby each point  $p\in Z$ one can pick a local holomorphic frame $e_1,\cdots,e_n$ for ${V}$ and a local
holomorphic coordinate $z_1,\cdots,z_n$ on $M$ to represent
$$\psi=h(z)(dz_1\wedge\cdots \wedge dz_n)\otimes(e_1\wedge\cdots \wedge e_n)
\and s=s_1(z)e_1+\cdots+s_n(z)e_n.$$
 In \cite[Chapter 5, Page 731]{GH} the Grothendieck residue is defined to be
$$\Res\frac{\psi}{s}=\sum_p\Res_p(\frac{\psi}{s})=\sum_p\Res_p\{\frac{h(z)dz_1\wedge\cdots \wedge dz_n}{s_1(z)\cdots s_n(z)}\}.$$

\black

Using such formulation \cite[Chapter 5]{GH} develops properties of residues over an arbitrary complex manifold $M$.
For example the residue theorem in complex analysis is generalized.
\begin{theo} \cite[Page 731]{GH}\label{Residue} If $M$ is compact, then
$\sum_{p\in Z} \Res_p(\frac{\psi}{s})=0.$\end{theo}

  \subsection{Virtual residue construction}
 Let us keep the symbol $Z$ to denote the zero loci of the fixed section of a bundle $V$ as before.
 If one gives a further thought about Assumption  \ref{GH-Setup}, it is not natural to assume that $Z$ is zero dimensional, even when $\rank V=\dim M$. It is easy to find an example where $Z$ has components of various dimensions.

 \begin{exam} Consider $V=\sO(2)\oplus \sO(2)$ over $\PP^2$ and
 $$s_t=(x_0\cdot x_1,(x_0+t(x_1-x_2))\cdot x_2)\in \Gamma(\PP^2,V),$$ where $[x_0,x_1,x_2]$ is the homogeneous coordinate of $\PP^2$.
 Let
 $$L_0=\{x_0=0\}\sub \PP^2\ , L_1=\{x_1=0\} \sub \PP^2\and L_2=\{x_2=0\}\sub \PP^2$$
 be three lines in $\PP^2$. Denote the zero loci of $s_t$ by $Z_t\sub \PP^2$. Then
\begin{enumerate}
\item $Z_0=L_0 \cup \{[1,0,0]\}$;
 \item $Z_t=\{[0,1,0],[1,0,0],[t,0,1],[0,1,1]\}$, \ \ \ \ \ for \ $t\neq 0$.
\end{enumerate}
 \end{exam}
\black

 As $Z$ can have positive dimensional components, we would still like to ask about residues associated to such locus. We begin with the following setup.



\noindent
 \begin{setup}\label{VR Setup}
 {\sl
(1). ${V}$ is a holomorphic bundle over complex manifold $M$\black,
$\dim M=\rank{V}=n;$\\
(2). $s\in\Gamma(M,V) $; the zero loci of $s$ is a compact set  $Z\sub  M$; \\
(3). $\psi\in \Gamma(M,K_M\otimes \det{V})$, called the ``weight".}   
\end{setup}

 We denote $U=M\setminus Z$, and  let $V_U$  be the restriction of $V$  over $U$.  
  Since $s$ is nowhere zero over $U$, the following Koszul sequence is exact over $U$
$$
    0\lra K_U\mapright{s} K_U\otimes V_U \mapright{ s\wedge} \cdots \mapright{ s\wedge}  K_U\otimes\wedge ^{n-1}V_U \mapright{ s\wedge} K_U\otimes\wedge^n V_U\lra 0.
$$

  The exact Koszul sequence induces a homomorphism
  \beq\label{ONE}
  \coh^0(U,K_U\otimes\wedge^n V_U)\lra \coh^{n-1}(U,K_U).
  \eeq
  One also has a canonical Dolbeault isomorphism
 \beq\label{TWO}
 \coh^{n-1}(U,K_U) \cong \coh_{\bar{\partial}}^{n,n-1}(U).
 \eeq
   Applying  \eqref{ONE} and \eqref{TWO} to $\psi$, and using that every $(n,n-1)$ form is $\partial$-closed, one obtains a (unique) De-Rham cohomology class
   \beq\label{eta}
   \eta\in \coh^{2n-1}(U,\CC).
  \eeq

For $Z$ is a compact analytic subset of $M$, by resolution of singularity, we can find birational transformation $\pi: \widetilde{M}\to M$, so that $\pi^{-1}(Z)$ is normal crossing divisors in $\widetilde{M}$ (see \cite[Theorem 2.0.3]{Wl}). Let $R_i\subset \widetilde{M}$ be the irreducible components. Let $T_i\subset \widetilde{M}$ be an $\epsilon_i$-tubular neighborhood of $R_i\subset \widetilde{M}$. Then $T_{\epsilon}:=\pi(\cup T_i)$ is a compact neighborhood of $Z$ in $M$,  where $T_{\epsilon}$
  has piecewise smooth boundary and $Z\hookrightarrow T_{\epsilon}$ is a homotopy equivalence. Such $T_{\epsilon}$ is called a ``good neighborhood" of $Z$ in $M$.

   Denote $N=\partial T_{\epsilon}$ be the boundary of $T_{\epsilon}$, whose  orientation   is  that induced from $M$. \black

 We then define the contribution of $Z$ to the residue associated from the datum $(U,M,V_U,s,\psi)$ to be
 \beq\label{res}
 \Res_{Z}\frac{\psi}{s}:= \frac{1}{(2\pi i)^n}\int_{N}\eta.
 \eeq

 If $T'$ is  another  good neighborhood  of $Z$,  by choosing smaller $\epsilon_i$ one can find a smaller good neighborhood $T$ of $Z$ such that $T\sub \mathring{T}_{\epsilon}\cap \mathring{T}'$. Then
$$
0=\int_{T_{\epsilon}\setminus \mathring{T}}\dbar\eta=\int_{T_{\epsilon}\setminus \mathring{T}}d\eta=\int_{N}\eta-\int_{\partial T}\eta
$$
and the similar identity for $T'$ imply $
\int_{N}\eta=\int_{\partial T'}\eta
$. Hence the definition is independent of the choice of good neighborhood.

It vanishes whenever $M$ is compact by Stoke theorem.
\begin{rema}\label{remark}
   For an arbitrary holomorphic $\psi\in\Gamma(U,K_M\otimes \det V)$, the same definition of residue is still valid.  \end{rema}
 \begin{exam}[Unorbid Landau Ginzburg B model of genus zero]
Suppose $M$ is a smooth quasi-projective variety over $\CC$ such that $\theta:\sO_M\to K_M^{\otimes 2}$ is an isomorphism of line bundles. Suppose
 $W:M\to\CC$ is a regular function (called ``superpotential") whose critical locus $(dW=0)\sub M$ is compact. We say $(M,W)$ is a Landau Ginzburg space.

 Let $V=\Omega_M$ and $s=dW\in\Gamma(M,V)$. For each $f\in\Gamma(M,\sO_M)$ the $\psi:=f\theta\in\Gamma(M, K_M\otimes \det V)$ associates a complex number
 $Res \frac{\psi}{s}=Res \frac{f\theta}{dW}$. This number is understood physically
  the correlator for the B-twisted LG model $(M,W)$ where $f$ is a given observable, as mentioned in Introduction. 
\end{exam}
 \begin{rema} The Grothendieck residue $\int_{\Gamma}\omega$ can be viewed as a period of the domain $M\setminus Z$ as it is an integral of a holomorphic form $\omega$ over an integral homology class $\Gamma\sub M\setminus Z$. 
 The authors are not aware of whether $\Res \frac{\psi}{s}$ (when $\dim(s=0)>0$) can be represented as some period of $M\setminus Z$.
\end{rema}

\section{Cohomology with compact support and Virtual residue}

In this section we represent the virtual residue as an integration of some compactly-supported twisted Dolbeault cohomology class. As before $V$ is a holomorphic bundle over a complex manifold $M$ with  $\text{rk}\, V=\dim M$ and $s$ is a holomorphic section of $V$ with compact zero loci $Z=(s=0)$.  Let $V^{*}$ be the dual bundle of $V$, and denote $\cA^{i,j}(\wedge^k V\otimes\wedge^l V^*)$ to be the sheaf of smooth $(i,j)$ forms
 on $M$ with value in $\wedge^k V\otimes \wedge^lV^*$.

 Denote $\Omega^{(i,j)}(\wedge^k V \otimes\wedge^l V^*):=\Gamma(M,\cA^{i,j}(\wedge^k V \otimes\wedge^l V^*))$
and assign its element $\alpha$ to have degree $\sharp \alpha=i+j+k-\ell$. Then
  $$\bB:=\oplus_{i,j,k,l}\Omega^{(i,j)}(\wedge^k V \otimes\wedge^l V^*)$$
is a graded commutative  algebra with the  (wedge) product uniquely extending wedge products in $\Omega^{\ast}, \wedge^\ast V, \wedge^\ast V^\ast$ and mutual tensor products. Denote
$$
E_M=\oplus_{0\leq i,j \leq n}E^{i,j}_M\subset \bB \ \ \ \ \ \
\mbox{with}\quad E^{i,j}_M:= \Omega^{(n,j)}(\wedge^i V)=\Gamma(M,\cA^{n,j}(\wedge^i V)),
$$
 and
$$
E_{c,M}=\oplus_{0\leq i,j \leq n}E^{i,j}_{c,M} \ \ \ \ \ \
\mbox{with}\quad E^{i,j}_{c,M}:= \{\alpha\in E^{i,j}_{M}| \ \alpha\   \text{has compact support}\}.
$$
 For $\alpha\in E_M$ we denote $\alpha_{i,j}$ to be its component in $E^{(i,j)}_M$. Clearly, $E_{M}$ is a bi-graded  $C^{\infty}(M)$-module. Under the operations
 $$\dbar: E^{i,j}_M\lra E^{i,j+1}_M \and s\wedge : E^{i,j}_M\lra E^{i+1,j}_M$$
  the space  $E^{\ast,\ast}_{M}$ becomes a double  complex and $E^{\ast,\ast}_{c,M}$ is a subcomplex. 
  We shall study the cohomology of $E_{c,M}$ with respect to the following  coboundary operator
  $$\dbar_s:=\dbar+s\wedge.$$
  One checks $\dbar_s^2=0$ using Leibniz rule of $\dbar$ and $\dbar s=0$.

 Let us introduce more operators. Fix a Hermitian metric $h$ on $V$.   Let
 $$\bar{s}:=\frac{(*,s)_h}{(s,s)_h}\in\Gamma(U,\cA^{0,0}( V^*)).$$ It associates a  contraction  map (c.f. \eqref{operator2'} in Appendix)   $$\iota_{\bar{s}}:\Gamma(U,\cA^{n,i}(\wedge^j V))\rightarrow\Gamma(U,\cA^{n,i}(\wedge^{j-1} V)).$$
 To distinguish it in later calculation, we denote
$
 \cT_s:= \iota_{\bar{s}}:   E^{*,*}_{U}\longrightarrow E^{*-1, *}_{U}.
$

The injection $j:U\to M$ induces the restriction  $j^*:E^{*, *}_M\to E^{*, *}_U$.
Let $\rho$ be a smooth cut-off function on $M$ such that $\rho|_{U_1}\equiv 1$ and $\rho|_{M\setminus U_2}\equiv 0$
 for some  relatively compact open neighborhoods $U_1\subset \overline{U}_1 \subset U_2$ of $Z$ in $M$.

 We define the degree of an operator to be its change on the total degree of elements in $E_M(E_U)$.
Then   $\dbar$ and $\cT_s$ are of degree $1$ and $-1$ respectively, and
 $[\dbar, \cT_s]=\dbar \cT_s+ \cT_s\dbar$ is of degree $0$.  Consider two  operators  introduced   in \cite[page 11]{LLS}
 \beq\label{defi-t1}
 T_\rho: E_M\to E_{c,M}  \qquad   \qquad T_\rho(\alpha):=\rho \alpha+(\dbar\rho)\cT_s {1\over 1+[\dbar, \cT_s]}(j^*\alpha) \eeq
and
\beq\label{defi-t2}
\ \   R_\rho: E_M\to E_M \qquad  \
 \qquad R_\rho(\alpha):= (1-\rho) \cT_s {1\over 1+[\dbar, \cT_s]}(j^*\alpha).
\eeq

Here as an operator
$$
{1\over 1+[\dbar, \cT_s]}:=\sum\limits_{k=0}^\infty(-1)^k[\dbar, \cT_s]^k
$$

  is well-defined since $[\dbar, \cT_s]^k(\alpha)=0$ whenever $k>n$. Clearly $T_\rho$ is of  degree zero and $R_\rho$ is of  degree by $-1$. Also $R_\rho(E_{c,M})\sub E_{c,M}$ by definition.

\begin{lemm}\label{lemmaforquasiiso}
  {\upshape $[\dbar_s,  R_\rho]=1-T_\rho$  } as operators on $E_M$.
\end{lemm}
\begin{proof}
It is direct to check that
\footnote{
As a notation convention, we always denote $[,]$ for the graded commutator, that is for operators $A, B$ of degree $|A|$ and $|B|$, the bracket is given by
$$
[A,B]=AB-(-1)^{|A||B|}BA.
$$  \black}
  \beq\label{commutator1} [s\wedge,  \cT_s]=1\quad \text{on}\  E_U.
  \eeq
Moreover,
       $$[P, [\dbar, \cT_s]]=0$$
 for   $P$ being   $s\wedge, \dbar$ or $\cT_s$. Therefore, we have
\begin{eqnarray*}
[\dbar_s,  R_\rho]&=&[\dbar_s, 1-\rho]\cT_s {1\over 1+[\dbar, \cT_s ]}j^*+(1-\rho)[\dbar_s,  \cT_s ] {1\over 1+[\dbar, \cT_s ]}j^* \\
&=&-(\dbar\rho)\cT_s \frac{1}{ 1+[\dbar, \cT_s ]}j^*+(1-\rho)j^*\\
&=&-(\dbar\rho)\cT_s \frac{1}{ 1+[\dbar, \cT_s ]}j^*+(1-\rho)=1-T_\rho.
\end{eqnarray*}
\end{proof}
\begin{prop}\label{quasi-isom}
The embedding $(E_{c,M}, \dbar_s\}\rightarrow \{E_M,\dbar_s)$ is a quasi-isomorphism.
\end{prop}

\begin{proof}
 By Lemma  \ref{lemmaforquasiiso} $
  H\sta(E_{M}/E_{c,M}, \dbar_s)\equiv0,$ and thus the proposition follows.
\end{proof}
We define the trace map  via integrating its component in $\Omega^{(n,n)}_M$, namely
$$
  \tr: E_{c,M}\to \mathbb{C},\ \qquad \tr(\alpha):=\int_M \alpha_{0,n}.
$$
 By definition  we have  $\tr(\dbar \alpha)=0$ and  $\tr(s\wedge\alpha)=0,$
which imply   that the trace map is well defined on the cohomology
\begin{eqnarray}\label{trace}
 \tr:  H\sta(E_{c,M}, \dbar_s) \to \mathbb{C}.
\end{eqnarray}
 Therefore Proposition \ref{quasi-isom} induces a trace map
$$  \tr:  H\sta(E_{M}, \dbar_s)\mapright{\cong} H\sta(E_{c,M}, \dbar_s) \to \mathbb{C}, $$
where the first isomorphism is the inverse of that induced from Proposition \ref{quasi-isom}.

\black

 \begin{prop}\label{compare} Let $V$ be a holomorphic bundle over a complex manifold $M$ with $\mbox{rk}\  V=\dim M=n$, and $s$ is a holomorphic section of $V$ with $Z=(s=0)$ compact. Let $\psi\in\Gamma(M,K_M\otimes\wedge^n V)$, then
 $[\psi]\in  H\sta(E_{M}, \dbar_s) $.
 Suppose $\psi=\alpha+\dbar_s \sigma$ for some $\alpha\in  H\sta(E_{c,M}, \dbar_s)$ and $\sigma\in  E_{M}$.
 Then
$$ \Res\frac{\psi}{s}=\frac{(-1)^{n}}{(2\pi i)^n}\tr(\psi)=\frac{(-1)^{n}}{(2\pi i)^n}\int_M \alpha_{0,n}.$$
 \end{prop}

\begin{proof}

Consider the Dolbeault  resolution of the Koszul exact sequence.
Recall the contraction map
$\cT_s:\Gamma(U,\cA^{n,q}(\wedge^p V))\rightarrow\Gamma(U,\cA^{n,q}(\wedge^{p-1} V))$ where one easily checks $(s\wedge \cT_s +\cT_s\ s\wedge )\alpha=\alpha$, for each $\alpha \in \Gamma(U,\cA^{n,q}(\wedge^p V))$.  Let
$$
\beta_0=\cT_s\psi, \ \beta_k=-\cT_s\bar{\partial}\beta_{k-1}=(-\cT_s\dbar)^k\cT_s\psi.
$$

By  \eqref{commutator1} we have
\begin{eqnarray*}
s\wedge \beta_0&=&\psi\\
s\wedge\beta_1&=&s\wedge (-\cT_s\dbar\beta_0)=-\dbar\beta_0+\cT_s\ s\wedge \dbar \beta_0\\
&=&-\dbar\beta_0-\cT_s\ \dbar\psi=-\dbar\beta_0\\
&\vdots&\\
s\wedge \beta_{n-1}&=&-\dbar\beta_{n-2}.
\end{eqnarray*}

 This implies $\dbar_s (\sum\beta_k)=\psi$. One obtains $\dbar\beta_{n-1}=0$
 because $\dbar \beta_{n-1}\in \Omega^{n,n}_U$ and  $$s\wedge\dbar \beta_{n-1}=\dbar(s\wedge\beta_{n-1})=-\dbar (\dbar\beta_{n-2})=0.$$

By zigzag the $\eta$ constructed in \eqref{eta} is identical to the class $[(-1)^{n-1}\beta_{n-1}]\in \coh^{(n,n-1)}(U)$. Thus by definition \eqref{res}
$$
\Res\frac{\psi}{s}=\frac{(-1)^{n-1}}{(2\pi i)^n}\int_{N} \beta_{n-1},
$$
 where  $N=\partial T$ for a good neighborhood $T$ of $Z$.  Let $T'$ be another good neighborhood of $Z$ in $M$, such that $T'\sub \mathring{T}$. \black Consider a smooth cut-off function $l$, which is zero on $T'$ and identical to one outside $T$. Set $\sigma'=l\sum_{k=0}^{n-1}\beta_k$ and $\alpha'=\psi-\dbar_s\sigma'$. Then
$\dbar_s \alpha'=\dbar_s\psi-\dbar_s\dbar_s\sigma'=0$ and
\begin{eqnarray*}
\alpha'&=&\psi-\dbar_s\sigma'\\
&=&\psi-\dbar(l\sum_{k=0}^{n-1}\beta_k)-s\wedge(l\sum_{k=0}^{n-1}\beta_k)\\
&=&\psi-(\dbar l)(\sum_{k=0}^{n-1}\beta_k)-l\dbar(\sum_{k=0}^{n-1}\beta_k)-s\wedge(l\sum_{k=0}^{n-1}\beta_k)\\
&=&\psi-(\dbar l)(\sum_{k=0}^{n-1}\beta_k)-l\psi\\
&=&(1-l)\psi-(\dbar l)(\sum_{k=0}^{n-1}\beta_k).
\end{eqnarray*}

 We then get a decomposition $\psi=\alpha'+\dbar_s \sigma'$. Because $\alpha'$ has compact support, by the definition of the trace map \eqref{trace}
\begin{eqnarray*}
\frac{1}{(2\pi i)^n}\int_M \alpha_{0,n}&=&\frac{1}{(2\pi i)^n}\int_M \alpha'_{0,n}\\
&=&-\frac{1}{(2\pi i)^n}\int_M \dbar (l\beta_{n-1})\\
&=&-\frac{1}{(2\pi i)^n}\int_M(\dbar l)\beta_{n-1}\\
&=&-\frac{1}{(2\pi i)^n}\int_{T\setminus \mathring{T'}}(\dbar l)\beta_{n-1}\\
&=&-\frac{1}{(2\pi i)^n}\int_{T\setminus \mathring{T'}} \dbar(l\beta_{n-1})\\
&=&-\frac{1}{(2\pi i)^n}\int_{N}\beta_{n-1}
\end{eqnarray*}
 using $\dbar (l\beta_{n-1})=d(l\beta_{n-1})$ and $l|_{\partial T'}=0$. Thus $
\frac{1}{(2\pi i)^n}\int_M \alpha_{0,n}=(-1)^{n}\Res\frac{\psi}{s}.
$
\end{proof}

\section{Exponential type Integral form of Virtual Residue}

  We aim to give a  natural integral representation
 for the virtual residue $\Res \frac{\psi}{s}$. 
   To do so we need to put metric on bundle and manifold and make suitable boundedness conditions. First let us
be in the situation where $V$ is a holomorphic bundle over a complex manifold $M$ with  $\text{rk}\, V=\dim M$ and $s$ is a holomorphic section of $V$ with compact zero loci $Z=(s=0)$.

  We  pick a reference point $\nu\in M$ and fix it once for all. We pick   a hermitian metric $h$ on $V$ and assume $M$ admits a complete Hermitian metric $g$ such that
there exists $C>0, \lam>1 $ making
\beq\label{bddvol}\text{vol}(B(r))\le Cr^\lam\qquad \forall \ r>0,\eeq
 where   $B(r):=\{z\in M|\, d(z,\nu)\le r\}$.  




  Denote $\cA^{i,j}(\wedge^k V\otimes\wedge^l V^*)$ to be the sheaf of smooth $(i,j)$ forms
 on $M$ valued in $\wedge^k V\otimes \wedge^lV^*$. The Hermitian metrics of $M$ and  $V$    induce a metric on the bundle which corresponds to the sheaf $\oplus_{i,j,k,l}\cA^{i,j}(\wedge^k V\otimes \wedge^lV^*)$ (c.f.   Appendix). \black Denote this metric by $(\cdot,\cdot)(z)$ for $z\in M$ and set $|\alpha|(z)=\sqrt{(\alpha,\alpha)(z)}$.

\begin{defi}\label{epq}
 We say $\alpha\in \Gamma(M,\cA^{i,j}(\wedge^k V\otimes\wedge^l V^*))$ is rapidly decreasing if for all $m\ge 0
$,
\begin{eqnarray*}
 \text{sup}_{z\in M}(1+d^2(z,\nu))^{m} |\alpha|(z)<\infty,
 \end{eqnarray*}
 where   $d(z,\nu)$ denotes the distance between $z$ and $\nu$.
\end{defi}

\begin{defi}\label{tempered}
We say $\alpha\in \Gamma(M,\cA^{i,j}(\wedge^k V\otimes\wedge^l V^*))$ is  tempered  if there exists an $m\ge 0$, such that
$$
\text{sup}_{z\in M}(1+d^2(z,\nu))^{-m}|\alpha|(z)< \infty,
 $$
 \end{defi}
 \begin{rema} By triangle inequality one may show both definitions are independent of the choice of the base point $\nu$, but we will not need this.
\end{rema}

We make the following assumption.
\begin{assu}\label{cond}

\begin{enumerate}
\item   The section $s$ is tempered;
\item   Let $\nabla$ be the Hermitian connection on $V$ with $\nabla^{0,1}=\dbar$. \black
\black The induced $\nabla s$ is tempered;
\item  There is a constant $C_0>0$ and a compact subset $Y$ of $M$ with $T\sub Y$, where $T$ is a good neighborhood of $Z$ in $M$ (c.f. (\ref{res})), such that $$
|s|^2(z)\ge C_0(1+d^2(z,\nu)),\ \   \forall z\in M\setminus Y.
$$
\end{enumerate}
\end{assu}

 In short the assumption says that $s$ has polynomial growth and $\nabla s$ has at most polynomial growth near $\partial M$.

\begin{rema}  If there is a holomorphic bundle $\sV$ over  a smooth complex projective variety $\sM$, and a section $\ti s$ of $\PP(\sV\oplus \sO_{\sM})$ such that $M=(\ti s\neq \infty)\sub \sM$ and $V=\sV|_M$, then one can construct $g,h$ satisfying \eqref{bddvol} and Assumption \ref{cond}. This provides a lot of examples. We omit the proof as it is not needed in this paper.  \end{rema}

\begin{lemm} \label{e}
If $s$ satisfies the above Assumption \ref{cond}, then $e^{-|s|^2}$ is rapidly decreasing, and $\beta\wedge\alpha$ is rapidly decreasing if $\beta\in \Omega^{(i,j)}(\wedge^k V)$ is tempered and $\alpha\in \Omega^{(l,\mu)}(\wedge^{\nu} V)$ is rapidly decreasing, or $\beta\in \Omega^{(i,j)}(\wedge^k V^*)$ is tempered and $\alpha\in \Omega^{(l,\mu)}(\wedge^{\nu} V^*)$ is rapidly decreasing.
\end{lemm}
\black
\begin{proof}
For arbitrary $m \ge 0$, we have
\begin{eqnarray*}
\text{sup}_{z\in M}(1+d^2(z,\nu))^{m} e^{-|s|^2}&\le &\text{max}\{\text{sup}_{z\in Y}(1+d^2(z,\nu))^{m} e^{-|s|^2},\\
 &&\text{sup}_{z\in M\setminus Y}(1+d^2(z,\nu))^{m} e^{-C_0(1+d^2(z,\nu))}\}\\
 &<&\infty,
\end{eqnarray*}
 by Assumption \ref{cond}.  Thus $e^{-|s|^2}$ is rapidly decreasing.
For $\beta\wedge \alpha$, there
a positive number $D$ such that for all $z\in M$
 \begin{eqnarray*}
 |\beta\wedge\alpha|(z)&\le & D\cdot |\beta|(z) |\alpha |(z).
\end{eqnarray*} \black
 Then for arbitrary $m\ge0$,
$$
\text{sup}_{z\in M}(1+d^2(z,\nu))^{m} |\beta\wedge\alpha|(z)< D\cdot \text{sup}_{z\in M}(1+d^2(z,\nu))^{m} |\beta|(z)|\alpha|(z)<\infty.
$$
  Thus  $\beta\wedge\alpha$ is rapidly decreasing.
\end{proof}


The contraction operator (defined in Appendix \eqref{operator2})  and the dbar operator
  $$\iota_s:\cA^{0,q}(\wedge^p V^*)\lra \cA^{0,q}(\wedge^{p-1} V^*),   \qquad \dbar:\cA^{0,q}(\wedge^p V^*)\lra \cA^{0,q+1}(\wedge^p V^*)$$
  define $\dbar+\iota_s$ that acts on 
 $$F_M^{p,q}\black =\Omega^{(0,q)}(\wedge^p V^*):=\Gamma(M,\cA^{0,q}(\wedge^p V^*)).$$
 Clearly $\oplus_{p,q} \Omega^{(0,q)}(\wedge^p V^*)$ is a graded subalgebra of $\bB$, and the action of $\dbar+\iota_s$ on $\oplus_{p,q} \Omega^{(0,q)}(\wedge^p V^*)$ satisfies Leibniz rule:
$$(\dbar+\iota_s)(\alpha\beta)=((\dbar+\iota_s)\alpha)\beta+(-1)^{\sharp \alpha}\alpha(\dbar+\iota_s)\beta.$$




Let $\psi$ be a holomorphic section of  $K_M\otimes\det V$. By using the contraction operator defined in Appendix \eqref{operator1}, we have the following map:
$$
\psi\lrcorner:\      F_M^{p,q}\lra E_M^{n-p,q}
\qquad \qquad\  \  u\mapsto \psi\lrcorner u.
$$

\begin{lemm} \label{con}
If $\psi$ is a tempered holomorphic section of  $K_M\otimes\det V$,
and $u\in F_M^{p,q}$ is rapidly decreasing, then
  $\psi\lrcorner u$ is also rapidly decreasing.
\end{lemm}
\begin{proof}
 By  Lemma \ref{inequa2}, there exists a constant $k$ such that for every $z\in M$
$$
(\psi\lrcorner u,\psi\lrcorner u)(z)\le k \cdot \black (u,u)(z)(\psi,\psi)(z).
$$
 Because $\psi$ is tempered   there exists   $m'\geq 0$ such that $|\psi|\le  C'(1+d^2(z,\nu))^{m'}$, where $C'$ is a constant. Then for arbitrary $m$
\begin{eqnarray*}
 \text{sup}_{z\in M}(1+d^2(z,\nu))^{m} |\psi\lrcorner u|(z)&\le & k\cdot  \text{sup}_{z\in M}(1+d^2(z,\nu))^{m}|u|(z)|\psi|(z)\\
 &\le &  k\cdot  C'\text{sup}_{z\in M}(1+d^2(z,\nu))^{m+m'}|u|(z)\\
 &< &\infty.
 \end{eqnarray*}
Thus $\psi\lrcorner u$ is rapidly decreasing.  \end{proof}

For $\beta\in \bB $, its exponential is defined as  $
e^{\beta}:=1+\beta+\frac{\beta^2}{2!}+\cdots,
$ which is a finite sum by degree reason.  Let $\xi=-(*,s)_h\in \Omega^{(0,0)}(V^{*})$. We define
\beq\label{S} S=(\dbar+\iota_s) \xi=-|s|^2+\dbar\xi\in\oplus_{p=0,1} \Omega^{(0,p)}(\wedge^p V^*).\eeq
Then   $e^S$ is an element in $\oplus_{p} \Omega^{(0,p)}(\wedge^p V^*)=\oplus_p F_M^{p,p}$.
\begin{lemm}\label{good}
If $s$ satisfies the Assumption \ref{cond}, then $\xi$, $\dbar\xi$ and $\dbar|s|^2$ are  tempered. $e^S \in \oplus_p F_M^{p,p}$ and $\dbar e^S \in \oplus_p F_M^{p,p+1}$ are both rapidly decreasing.
\end{lemm}\black
\begin{proof}
Let $z\in M$ be an arbitrary point. Then by formula 1.20 in page 63 of  \cite{Wu}, there exists a local holomorphic frame $\{e_i\}$ of $V$ around $z$, such that for any $i,j$
$$
(e_i,e_j)(z)=\delta_i^j \and
\nabla e_i(z)=0.
$$
Let $e^j$ be the dual local frame of $V^*$. Locally represent $s=\sum_i s_i e_i$, so $\xi(z)=-\sum_i\overline{s_i}e^i$. \black Then
$$
(\xi,\xi)(z)=\sum_i|s_i|^2=(s,s)(z).
$$
 Varying  $z\in M$ we see $\xi$ is tempered because $s$ is tempered.
Then
$$(\nabla s)(z)=\sum_i (ds_ie_i+s_i\nabla e_i)(z)=\sum_i (ds_i)(z)e_i$$
implies
$$
(\nabla s,\nabla s)(z)=(\sum_i (ds_i)(z)e_i,\sum_j (ds_j)(z)e_j)=\sum_i(ds_i,ds_i)(z).
$$
 Hence
\begin{eqnarray*}
\dbar \xi(z)  &  =   &   -\dbar [ \sum_i(e_i,s) e^i ] (z)=-\sum_i\dbar (e_i,s) (z) e^i\\
&=&-\sum_i(e_i,\nabla s)(z)e^i =-\sum_i   (e_i,\sum_jds_je_j) \black (z)e^i=-\sum_i\overline{ds_i}(z)e^i.
\end{eqnarray*}
Hence we have
$$
(\dbar \xi,\dbar \xi)(z)=(\sum_i\overline{ds_i}(z)e^i,\sum_i\overline{ds_j}(z)e^j)=\sum_i(\overline{ds_i},\overline{ds_i})(z) =\sum_i(ds_i,ds_i)(z)= (\nabla s,\nabla s)(z).
$$
\black
 Varying $z\in M$ we see $\dbar\xi$ is tempered because $\nabla s$ is tempered by assumption.

 Therefore  arbitrary power of $\dbar\xi$ is also tempered. By Lemma \ref{e},   $e^{-|s|^2}$ is rapidly decreasing.  Therefore using Lemma \ref{e}
$$
e^{S}=e^{-|s|^2}(1+\dbar\xi+\frac{(\dbar\xi)^2}{2!}+\cdots+\frac{(\dbar\xi)^n}{n!})\in \oplus_p\Omega^{(0,p)}(\wedge^p V^*)
$$
is rapidly decreasing.

Using the formula \eqref{kob} in Appendix, we have  $$\dbar |s|^2(z)=-\dbar<s,\xi>(z)=<s,\dbar\xi>.$$

  Therefore one has
$(\dbar |s|^2,\dbar |s|^2)\le (s,s)(\dbar\xi,\dbar\xi).
$
 Then  $\dbar |s|^2$ is tempered as   $s$ and $\dbar\xi$ are tempered. Apply Lemma \ref{e} again
$$
\dbar e^S=-\dbar|s|^2 e^{-|s|^2}(1+\dbar\xi+\frac{(\dbar\xi)^2}{2!}+\cdots+\frac{(\dbar\xi)^n}{n!})
$$
is rapidly decreasing, for $\dbar |s|^2$ and $\dbar\xi$ are tempered and $e^{-|s|^2}$ is rapidly decreasing.
\end{proof}

\begin{lemm}\label{integral1}
Let $s$ and $Y$ be in the Assumption \ref{cond}, and $\cT_s$ is defined in section 3. Suppose $\psi$ is a tempered holomorphic section of  $K_M\otimes\det V$, and $\rho$ is a smooth function with compact support.  One can find positive constants $\mu$ and $C_1$ so that $$|(\dbar\rho) \cT_s(\dbar \cT_s)^k(\psi\lrcorner e^S)|(z)\le    C_1|\dbar\rho|e^{-|s|^2}(1+d^2(z,\nu))^{\mu}(z),\ \ \ \forall z\in M\setminus Y.$$
\end{lemm}
\begin{proof}  By definition $e^S$ can be written as $e^S=\sum_i w_i$, where $w_i=e^{-|s|^2}\frac{(\dbar\xi)^i}{i!}\in F_M^{i,i}$. First we claim for arbitrary $k\in\NN$ one has
\beq\label{ind}
(\dbar \cT_s)^k(\psi\lrcorner w_i)=\psi\lrcorner((\dbar \bar{s})^k\wedge w_i)-\psi\lrcorner(\bar{s}\wedge(\dbar\bar{s})^{k-1}\wedge\dbar w_i).
\eeq

We prove it by induction. For $k=1$, by Lemma \ref{sign} and Lemma \ref{sign1}
$$ \dbar \cT_s(\psi\lrcorner w_i)=\dbar(\psi\lrcorner (\bar{s}\wedge w_i))
=\psi\lrcorner (\dbar\bar{s}\wedge w_i)-\psi\lrcorner (\bar{s}\wedge\dbar w_i).$$
Assuming \ref{ind} holds for $k=l-1$, then
\begin{eqnarray*} \dbar \cT_s(\dbar \cT_s)^{l-1}(\psi\lrcorner w_i)&=&\dbar \cT_s(\psi\lrcorner((\dbar \bar{s})^{l-1}\wedge w_i)-\psi\lrcorner(\bar{s}\wedge(\dbar\bar{s})^{l-2}\wedge\dbar w_i))\\
&=&\dbar(\psi\lrcorner(\bar{s}\wedge(\dbar \bar{s})^{l-1}\wedge w_i)\\
&=&\psi\lrcorner((\dbar \bar{s})^l\wedge w_i)-\psi\lrcorner(\bar{s}\wedge(\dbar\bar{s})^{l-1}\wedge\dbar w_i).
\end{eqnarray*}

This proves the claim.  Therefore
\begin{eqnarray*}
\cT_s(\dbar \cT_s)^k(\psi\lrcorner w_i)&=&\cT_s(\psi\lrcorner((\dbar \bar{s})^k\wedge w_i)-\psi\lrcorner(\bar{s}\wedge(\dbar\bar{s})^{k-1}\wedge\dbar w_i))\\
&=&\psi\lrcorner(\bar{s}\wedge(\dbar \bar{s})^k\wedge w_i).
\end{eqnarray*}
 Over $M\setminus Z$ we have  $\bar{s}=-\frac{\xi}{|s|^2}$ and thus an identity
\begin{eqnarray*}
\bar{s}\wedge(\dbar \bar{s})^k\wedge w_i&=&(-\frac{\xi}{|s|^2})\wedge(-\frac{\dbar\xi}{|s|^2}+\frac{(\dbar|s|^2)\xi}{|s|^4})^k\wedge w_i\\
&=&(-\frac{\xi}{|s|^2})\wedge(-\frac{\dbar\xi}{|s|^2})^k\wedge w_i\\
&=&(-1)^{k+i+1}(i!)^{-1}|s|^{-2(k+1)}e^{-|s|^2}\xi(\dbar\xi)^{k+i}
\end{eqnarray*}

 Assumption \ref{cond} (3) implies $|s|^2(z)\ge C_0$ for $z\in M\setminus Y$.  Since $\xi$ and $\dbar\xi$ are tempered  by Lemma \ref{good},  so is $\xi(\dbar\xi)^l$. By Lemma \ref{inequa2} there exists a positive number $C'$ independent of $z\in M\setminus Y$, such that
\begin{eqnarray*}
|(\dbar\rho) \cT_s(\dbar \cT_s)^k(\psi\lrcorner e^S)|(z)&\le&\sum_i|(\dbar\rho)\wedge \bl\psi\lrcorner(\bar{s}\wedge(\dbar \bar{s})^k\wedge w_i)\br|(z) \\
&\le& C' \sum_i|\psi||\dbar\rho||\bar{s}\wedge(\dbar \bar{s})^k\wedge w_i|(z)\\
&\le& C_1|\dbar\rho|e^{-|s|^2}(1+d^2(z,\nu))^{\mu}(z).
\end{eqnarray*}   for some positive $\nu$ and $C_1$ independent of $z\in M\setminus Y$ \black
\end{proof}

\begin{lemm}\label{exact-closed}
 $e^{S}$ is $(\dbar+\iota_s)$ -closed, and $1-e^{S}$ is $(\dbar+\iota_s)$-exact.
 \end{lemm}

\begin{proof}
The first assertion is from $(\dbar+\iota_s)^2=0$ and the Leibniz rule of $\dbar+\iota_s$. Use  the similar identity $(\dbar+\iota_s)e^{tS}=0$ ($t$ is a variable) and $e^{ x}-1=\int_0^1 xe^{ tx} dt, $
we have
 $$e^{ S}-1=(\bar{\partial}+\iota_{s})\int_0^1 \xi e^{ tS} dt,$$
 and the exactness follows.
 \end{proof}


\black



 Since $e^S$ lies in $\oplus_p\Omega^{(0,p)}(\wedge^p V^*)$, the objects $\psi\lrcorner e^S$ (and hence $\psi\lrcorner(1- e^S)$) lie in $\oplus_p\Omega^{(n,p)}(\wedge^{n-p} V)$, a subspace of $E_M$ defined in previous section.

\begin{lemm}\label{verygood}
If $\psi$ is a holomorphic section of  $K_M\otimes\det V$, then  $\psi\lrcorner(1-e^S) $ is  $\dbar_s$-exact, and $\psi\lrcorner e^{S}$ is $\dbar_s$-closed.

\end{lemm}
\begin{proof}
 Denote $e^{tS}=e^{-t|s|^2}\sum_{k=0}^nt^k\beta_k$, where $\beta_k=\frac{(\dbar\xi)^k}{k!}\in F_M^{k,k} $.  By  previous Lemma one may represent   $1-e^S=(\dbar+\iota_s)\omega$ with
 $$\omega=\sum\omega_{p,q}, \qquad \omega_{p,q}\in F_M^{q,p}$$
  where the sum runs over integers $p,q\in [0,n]$.
   By  Lemma \ref{sign} and Lemma \ref{sign1}, we have
$$\psi\lrcorner (\dbar w_{p,q})= \dbar (\psi\lrcorner w_{p,q}) \and
 \psi\lrcorner (\iota_s w_{p,q})= s\wedge (\psi\lrcorner w_{p,q}) .$$
Together we obtain
$$\psi\lrcorner (1-e^S)
=\sum  \psi\lrcorner [(\dbar +\iota_s)  \omega_{p,q}]=  (\dbar+s\wedge)[\sum \psi\lrcorner \omega_{p,q}]
$$
is exact with respect to the operator  $\dbar_s:=\dbar+s\wedge$.

 Using $\psi=\psi\lrcorner e^{S}+\psi\lrcorner(1-e^S)$ and that $\psi$ is $\dbar_s$-closed, we have
$$
\dbar_s(\psi\lrcorner e^{S})=\dbar_s(\psi-\psi\lrcorner(1-e^S))=0
$$
\end{proof}

 \begin{lemm}\label{integral} For each rapidly decreasing $\alpha\in E^{\ast,\ast}_M$ one has
  $$|\int_{M}\alpha _{0,n}|\le \int_M|\alpha |d\text{vol}_M<\infty.$$

 \end{lemm}

\begin{proof} By definition of rapidly decreasing, there exists a constant $D$ and an $l>\lam+2$ ($\lam$ is as in \eqref{bddvol}) such that
$$|\alpha|<D (1+d^2(z,\nu))^{-l} \qquad ,\forall z\in M.$$
Hence
$$|\int_{M}\alpha _{0,n}|\le\int_M|\alpha_{0,n}|d\text{vol}_M\le \int_M|\alpha| d\text{vol}_M\le D\int_M (1+d^2(z,\nu))^{-l}d\text{vol}_M.$$
 Recall $B(\rho):=\{z\in M| d(z,\nu)\le \rho\}$. By \eqref{bddvol} and completeness of $g$ on $M$
\begin{eqnarray*}
\int_M (1+d^2(z,\nu))^{-l}d\text{vol}_M&=&\lim_{k\rightarrow \infty}\int_{B(k)}(1+d^2(z,\nu))^{-l}d\text{vol}_M\\
&=&\sum_k\int_{B(k)\setminus B(k-1)}(1+d^2(z,\nu))^{-l}d\text{vol}_M\\
&\le&\int_{B(1)}(1+d^2(z,\nu))^{-l}d\text{vol}_M+C\sum_{k=2}^{\infty}(1+(k-1)^2)^{-l}k^\lam\\
&<&\infty.
\end{eqnarray*}
 where the last series converges because $l>\lam+2$. \black  This proves the claim.\end{proof}

  Using Proposition \ref{compare} we obtain an exponential type integral presentation of virtual residues.
\begin{theo}
\label{4C} Suppose $s$ satisfies Assumption \ref{cond}.
 Then for each tempered holomorphic section $\psi$ of  $K_M\otimes\det V$, the contraction  $\psi\lrcorner e^{S}\in    \oplus_{p+q=n} \Omega^{n,q}(\wedge^p V)$ is rapidly decreasing, one has
\begin{eqnarray}
\Res\frac{\psi}{s}=\frac{(-1)^n}{(2\pi i)^n}\int_M \psi\lrcorner e^{S}.
\end{eqnarray}
\end{theo}

\begin{proof}  By Lemma \ref{con} and Lemma \ref{good}  $\psi\lrcorner e^{S} $ is rapidly decreasing.  By the completeness of the metric, there exists an exhaustive sequence of compact subsets $\{K_{i}\}$ of $M$, $M=\cup K_i$, and smooth functions $\rho_i$ such that
 \begin{eqnarray*}
 \rho_i=1 \ \ \ \ \ \ \text{in a neighborhood of } \ \ K_{i}, \ \ \ \text{Supp}\, \rho_i\subset K^\circ_{i+1}\\
 0\le\rho_i\le 1\ \ \ \ \text{and}\ \ \ \ |d\rho_i|\le 2^{-i},
 \end{eqnarray*}
see Lemma 2.4 in page 366 of \cite{Dem}.
Choosing $c$ big enough such that $Z\sub Y\subset K_{c}$, where $Y$ is compact as in Assumption \ref{cond}, then $M=\cup_{j\ge c} K_{j}$.
 By the definition of $T_{\rho}$ in \eqref{defi-t1} and Lemma \ref{lemmaforquasiiso}, we have
 $$
 [\dbar_s,R_{\rho_j}](\psi\lrcorner e^{S})=\psi\lrcorner e^{S}-T_{\rho_j}(\psi\lrcorner e^{S})
 $$
 and pointwise convergence
\begin{equation}\label{lim}
\lim_{j\to\infty} T_{\rho_j}(\psi\lrcorner e^{S})=\psi\lrcorner e^{S}.
\end{equation}
Thus we may write
$$\psi=T_{\rho_j}(\psi\lrcorner e^{S})+\dbar_s (R_{\rho_j}(\psi\lrcorner e^{S}))+ \psi\lrcorner (1-e^S) ,$$
where  $\psi\lrcorner (1-e^S)$ is also $\dbar_s$-exact by Lemma \ref{verygood}
and $T_{\rho_j}(\psi\lrcorner e^{S})$ is compactly supported by definition \eqref{defi-t1}.
 Apply Proposition \ref{compare}
$$
\int_M T_{\rho_j}(\psi\lrcorner e^{S})=(-2\pi i)^n\Res\frac{\psi}{s}.
$$

 Using  constants $\mu,C_1$ in Lemma \ref{integral1}, we define a smooth positive function on $M$
  $$G(z)=|\psi\lrcorner e^S|+(n+1) C_1e^{-|s|^2}(1+d^2(z,\nu))^{\mu}(z).$$

 By Lemma \ref{con} and Lemma \ref{good}, $\psi\lrcorner e^{S}$ is rapidly decreasing.   Because $e^{-|s|^2}$ is rapidly decreasing and $(1+d^2(z,\nu))^{\mu}$ is tempered, we know $G(z)$ is also rapidly decreasing. \black  By Lemma \ref{integral}
  $$
\int_M G(z)\text{dvol}_M<\infty.
$$    Therefore $G(z)\in L^1(M)$, where $L^1(M)$ is the function space with the norm  $||\beta||:=\int_M|\beta|d\text{vol}_M$ (c.f. the definition in \cite[page 288]{Dem}).


 Recall in  definition \eqref{defi-t1} $\cT_s[\dbar,\cT_s]^k=\cT_s(\dbar \cT_s)^k$ because $\cT_s^2=0$. Therefore
 $$T_{\rho_j}(\psi\lrcorner e^{S})=\rho_j(\psi\lrcorner e^{S})+(\dbar\rho_j)\sum_{k=0}^n (-1)^k \cT_s(\dbar \cT_s)^k(\psi\lrcorner e^S).$$

  Take absolute value and use Lemma \ref{integral1} one sees at  arbitrary $z\in M\setminus Y$\black
\begin{eqnarray*}
|T_{\rho_j}(\psi\lrcorner e^{S})|
&\le&|\psi\lrcorner e^{S}|+(n+1) C_1|\dbar\rho_j|e^{-|s|^2}(1+d^2(z,\nu))^{\mu}(z)\le G(z).
\end{eqnarray*}
When $z\in Y$ one has $\dbar \rho_j(z)=0$ because $Y\sub K_c$. Thus the same inequality holds for arbitrary $z\in M$. Then by Lemma \ref{integral},   \eqref{lim} and Lebesgue dominated convergence theorem \cite[page 376]{Roy} \black
,  we have
 $$
\int_M(\psi\lrcorner e^{S})=\lim_{j\to\infty}\int_M T_{\rho_j}(\psi\lrcorner e^{S}) \black =(-2\pi i)^n\Res\frac{\psi}{s}.
 $$
\end{proof}

\begin{prop}\label{inde}
	Suppose $\psi$ is a tempered holomorphic section of  $K_M\otimes\det V$, and $s$ satisfies the Assumption \ref{cond}. Then 
	we have that $\int_M (\psi\lrcorner e^{tS})$ is  independent of $t$ for $t>0$.
\end{prop}

\begin{proof} 
By Lemma \ref{con} and Lemma \ref{good}, $\psi\lrcorner e^{tS}$ is rapidly decreasing for $t>0$. Therefore
$\int_M (\psi\lrcorner e^{tS})$ is finite for $t>0$ by Lemma \ref{integral}.
  
Applying Leibniz rule to $\dbar +\iota_s$ and Lemma \ref{exact-closed} one has
	$$
	\frac{d(\psi\lrcorner e^{tS})}{d t}=\psi\lrcorner (S e^{tS})=\psi\lrcorner (((\dbar+\iota_s)\xi) e^{tS}) =\psi\lrcorner ((\dbar+\iota_s)(\xi e^{tS})).
	$$
	  Lemma \ref{sign} and Lemma \ref{sign1} imply that
	$$
	\psi\lrcorner ((\dbar+\iota_s)(\xi e^{tS}))
	=\dbar_s(\psi\lrcorner (\xi e^{tS})).$$
	Note that here $\psi\lrcorner (\xi e^{tS})$ is in $\oplus_q \Omega^{(n,q)}(\wedge^{n-1-q}V^*)$.

	By Lemma \ref{good}, $\xi,\dbar\xi$ are  tempered, and $e^{tS}$ is rapidly decreasing. Then Lemma \ref{e} implies $\xi e^{tS}$ is rapidly decreasing and that
	$$(\dbar+\iota_s)(\xi e^{tS})=((\dbar+\iota_s)(\xi)) e^{tS}=(\dbar\xi) e^{tS}-|s|^2e^{tS}$$
	is also rapidly decreasing. Therefore Lemma \ref{con} shows  $\psi\lrcorner(\xi e^{tS})$ and $\psi\lrcorner((\dbar+\iota_s)(\xi e^{tS}))$ are rapidly decreasing. Using Lemma \ref{integral}, we have $$\int_{M}|\psi\lrcorner(\xi e^{tS})|d\text{vol}_M<\infty \and \int_{M}|\psi\lrcorner((\dbar+\iota_s)(\xi e^{tS}))|d\text{vol}_M<\infty.$$  Therefore for  $\varpi$ to be the
	component of $\psi\lrcorner (\xi e^{tS})$ in $\Omega^{(n,n-1)}$ one has
	$$\int_M|\varpi|d\text{vol}_M\leq \int_{M}|\psi\lrcorner(\xi e^{tS})|d\text{vol}_M<\infty,$$
	
	$$\int_M |d\varpi|d\text{vol}_M=\int_M |\dbar\varpi|d\text{vol}_M\leq
	\int_M |\dbar_s (\psi\lrcorner (\xi e^{tS})    ) |d\text{vol}_M=
	 \int_{M}|\psi\lrcorner((\dbar+\iota_s)(\xi e^{tS}))|d\text{vol}_M<\infty.$$
	
	Apply \cite[p141 Thm]{Gaf} one concludes $\int_M d\varpi=0$, and thus $\int_M d(\psi\lrcorner (\xi e^{tS}))=0$. Using that  $s\wedge(\psi\lrcorner (\xi e^{tS}))$ has no component in $\Omega^{(n,n)}$,  we have 
		$$   \int_M \frac{d( \psi\lrcorner  e^{tS}) }{dt} =\int_M\dbar_s(\psi\lrcorner( \xi e^{tS}))=\int_M\dbar\varpi=\int_M d\varpi=0.$$

 We finally claim $\frac{d\int_M( \psi\lrcorner  e^{tS}) }{dt}=\int_M \frac{d( \psi\lrcorner  e^{tS}) }{dt}$. One first writes $e^{tS}=e^{-t|s|^2}\sum_{k=0}^nt^k\beta_k$ with $\beta_k=\frac{(\dbar\xi)^k}{k!}\in F^{k,k}_{M} $. Because $e^{-t|s|^2}\beta_k$ is rapidly decreasing and $\dbar \xi,|s|^2$ are tempered,  $e^{-t|s|^2}(\dbar\xi)\beta_{k}$ and $e^{-t|s|^2}|s|^2\beta_{k}$ are rapidly decreasing. As $\psi\lrcorner$
  preserves the rapidly-decreasing property,  there exists a constant $C_2$ such that, for $t>t_0>0$
\begin{eqnarray*}
|\frac{d(\psi\lrcorner e^{tS})}{d t}|&=&e^{-t|s|^2}|-\sum_{k=0}^nt^{k}|s|^2
(\psi\lrcorner\beta_k)
+\sum_{k=1}^nt^{k-1}\psi\lrcorner((\dbar\xi)\beta_{k-1})|\\
&\le&C_2  e^{-\frac{t}{2}|s|^2} \le C_2  e^{-\frac{t_0}{2}|s|^2},
\end{eqnarray*}
where $e^{-\frac{t_0}{2}|s|^2}$ is in $L^1(M)$ (integrable). This implies $$\frac{d\int_M\psi\lrcorner  e^{tS} }{dt} =  \int_M \frac{d( \psi\lrcorner  e^{tS}) }{dt}=0.$$
\end{proof}

\begin{coro} Under the same condition and for $t>0$ let  $s_t=t\cdot s$
associates $\xi_t,S_t$ as in \eqref{S}, then $\int_M \psi\lrcorner e^{S_t}=t^{-n}\int_M \psi\lrcorner e^S$.
\end{coro}
\begin{proof} One checks $(\psi \lrcorner e^{S_t})_{n,n}=t^{-n}(\psi\lrcorner e^{t^2 S})_{n,n}$, and then applies the previous proposition.
\end{proof}

  The identity corresponds to $Res\frac{\psi}{ts}=\frac{1}{t^n}Res\frac{\psi}{s}$ for $n=\dim M$ (residue taken near compact connected components of $(s=0)$).

 \section{Appendix:  operators and metrics on exterior algebra $\bB$}

 Let $V$ be a rank $n$ holomorphic bundle over a complex manifold $M$.  Recall in section three  $\bB:=\oplus_{i,j,k,l}\Omega^{(i,j)}(\wedge^k V \otimes\wedge^l V^*)$ is a graded commutative algebra  extending the wedge products of $\Omega^{\ast}, \wedge^\ast V$ and $\wedge^\ast V^\ast$. 
 The degree of $\alpha\in \Omega^{(i,j)}(\wedge^k V \otimes\wedge^l V^*)$ is   $\sharp\alpha:=i+j+k-l$.  We brief $A^0(\wedge^k V \otimes\wedge^l V^*)=\Omega^{(0,0)}(\wedge^k V \otimes \wedge^l V^*)$.\black 

   Set  $\kappa:\bB\to \Omega\sta$
which sends $\omega (e\otimes e')$(for $\omega\in\Omega^{(i,j)}, e\in \wedge^k V, e'\in \wedge^\ell V\sta)$ to $\omega \langle e,e'\rangle$, where $\langle,\rangle$
is the dual pairing between $\wedge^k V,\wedge^k V\sta$  and $\langle e,e'\rangle =0$ when $k\neq \ell$. We further extend the pairing by setting  $\langle\alpha,\beta\rangle :=\kappa(\alpha \beta)$  for $\alpha,\beta\in \bB$.
It is direct to verify \beq\label{kob}\dbar\langle\alpha,\beta\rangle =\langle\dbar\alpha,\beta\rangle +(-1)^{\sharp \alpha}\langle\alpha,\dbar\beta\rangle .\eeq

 We now define three different types of contraction maps.
Given   $u\in \Omega^{(i,j)}(\wedge^k V)$ and
 $k\geq \ell$,  we   define
 \beq\label{operator1} u \lrcorner: \Omega^{(p,q)}(\wedge^{l}V^*)\lra \Omega^{(p+i,q+j)}(\wedge^{k-l}V) \eeq
where for $\theta\in \Omega^{(p,q)}(\wedge^{\ell}V^*)$, the $u\lrcorner \theta$ is determined by
$$\langle u\lrcorner \theta,\nu\sta\rangle =(-1)^{(i+j)l+(p+q)\sharp u+\frac{l(l-1)}{2}}\langle u,\theta\wedge \nu^*\rangle ,\qquad \forall \nu^*\in A^0(\wedge^{k-l}V^*).$$

 Given $\alpha\in A^0(V)$, we define
\beq\label{operator2}
\iota_{\alpha} : \Omega^{(i,j)}(\wedge^{k}V^*) \lra \Omega^{(i,j)}(\wedge^{k-1}V^*)\eeq
where for $w\in  \Omega^{(i,j)}(\wedge^{k}V^*)$, the $\iota_{\alpha}(w)$ is determined by
 $$\langle \nu,\iota_{\alpha}(w)\rangle =\langle \alpha\wedge \nu,w\rangle ,\qquad \forall \nu \in A^0(\wedge^{k-1}V).$$

 For above $\alpha$, $\theta$ and $w$ one has
 $\iota_\alpha(w\wedge \theta)=\iota_\alpha(w)\wedge\theta +(-1)^{\sharp w}w\wedge \iota_\alpha(\theta).$
 Given $\gamma\in A^0(V^*)$, we also  define
\beq\label{operator2'}
\iota_{\gamma} : \Omega^{(i,j)}(\wedge^{k}V) \lra  \Omega^{(i,j)}(\wedge^{k-1}V),\eeq

 for $v\in  \Omega^{(i,j)}(\wedge^{k}V)$, the $\iota_{\gamma}(v)$ is determined by
 $$\langle \iota_{\gamma}(v),w\rangle =(-1)^{i+j}\langle v,\gamma\wedge w\rangle , \qquad \forall w\in A^0(\wedge^{k-1}V^*).$$

 We have the following identities. Because the proofs of the identities are standard, we omit them here.\black

\begin{lemm}\label{sign}
Given $u\in \Omega^{(i,j)}(\wedge^n V)$,  and $\theta,\alpha, \gamma$ as above, one has
$$
\alpha \wedge(u\lrcorner \theta)=u\lrcorner(\iota_\alpha (\theta) )\and
\iota_{\gamma}(u\lrcorner \theta)=u\lrcorner(\gamma\wedge \theta).
$$

\end{lemm}

\begin{lemm}\label{sign1}
 For   $u\in\Omega^{(i,j)} (\wedge^k V), \theta\in \Omega^{(p,q)} (\wedge^\ell V\sta)$, $k\ge l$ and smooth form $\alpha\in\Omega^{(a,b)}(M)$, we have
 $$\alpha \wedge (u\lrcorner  \theta)= u\lrcorner (\alpha \theta)\and   \dbar (u\lrcorner \theta)=(-1)^{\sharp \theta}(\dbar u)\lrcorner \theta + u \lrcorner (\dbar \theta).$$
\end{lemm}

  Now we study some simple metric inequalities on $\bB$. Let $h$ be a fixed hermitian metric over $V$. For arbitrary holomorphic local frame  $\{e_i\}$  of $V$  with $\{t^i\}$ its dual frame of $V^*$,  one represents $h=\sum h_{i\bar{j}}t^i\otimes\bar{t}^j$. The induced metric $h^*$ on $V^*$ can be written as $h^*=\sum h^{i\bar{j}}e_i\otimes \bar{e}_{j}$, where $\sum h^{i\bar{k}}h_{j\bar{k}}=\delta_i^j$.

 As in \cite[page 79 Ex 13]{War}, the induced  metric $h_{\wedge^k V}$ on $\wedge^k V$ is
 $$
 (\alpha_1\wedge\cdots\wedge\alpha_k,\beta_1\wedge\cdots\wedge\beta_k)_{h_{\wedge^k V}}:=\det [h(\alpha_i,\beta_j)]
 $$


Similarly $h^*$ induced metrics $h^*_{\wedge^l V^*}$ on $\wedge^l V^*$ and $h_{\wedge^k V}\otimes h_{\wedge^l V}^*$   on $\wedge^k V\otimes \wedge^l V^*$.  
 The induced metric on $\bB=\oplus_{i,j,k,l}\Omega^{(i,j)}(\wedge^k V \otimes\wedge^l V^*)$ would be denoted by $(\cdot,\cdot)$ and $|\alpha|^2:=(\alpha,\alpha)$ for $\alpha\in \bB$. We have the following inequality,

\begin{lemm}\label{inequa2}
For  $u\in \Omega^{(n,0)}(\wedge^k V)$ and  $v^*\in \Omega^{(0,q)}(\wedge^l V^*)$ with $k\ge l$,  one has
\begin{eqnarray*}
	(u\lrcorner v^*,u\lrcorner v^*)&\le &  b(u,u) (v^*,v^*),
\end{eqnarray*}
where $b$ depends on the  ranks of the bundles correspond to $\Omega^{(0,q)}\otimes\wedge^l V^*, \wedge^{k-l}V^*$.
\end{lemm}

\bibliographystyle{amsplain}

\begin{thebibliography}{10}



\bibitem[1]{Bru} Bruzzo U, Rubtsov V. On localization in holomorphic equivariant cohomology. Central European Journal of Mathematics, {\bf 10}(4), 1442-1454 (2012)






\bibitem[2]{CL1} Chang H L, Li J. Gromov-Witten invariants of stable maps with fields.
Int. Math. Res. Not.  {\bf 18}, 4163--4217 (2012)


\bibitem[3]{CL2} Chang H L, Li J. An algebraic proof of the hyperplane property of the genus one GW-invariants of quintics.  J. Diff. Geom. {\bf 100}(2), 251-299 (2015). arxiv:1206.5390




\bibitem[4]{CLL} Chang H L, Li J,  Li W P.  Witten's top Chern classes via cosection localization. Invent. Math. {\bf 200}(3), 1015-1063   (2015)

\bibitem[5]{CLLL1} Chang H L, Li J,  Li W P, Liu C C.
  Mixed-Spin-P fields of Fermat quintic polynomials. arXiv:1505.07532

\bibitem[6]{CLLL2} Chang H L, Li J,  Li W P, Liu C C.
 Toward an Effective Theory of GW Invariants of Quintic Threefolds.  arXiv:1603.06184




 
 \bibitem[7]{Ch} Chiodo A. A construction of Witten's top Chern class in K-theory. Gromov-Witten theory of spin curves and orbifolds. Contemp. Math. vol. 403 (2006),  pp. 21-29, Amer. Math. Soc., Providence  (2006)


\bibitem[8]{CR} Chiodo A, Ruan Y B. Landau-Ginzburg/Calabi-Yau correspondence for quintic three-folds via symplectic transformations. Invent. Math. {\bf 182}(1), 117-165 (2010)


\bibitem[9]{Dem} Demailly J P. Complex Analytic and Differential Geometry. (Universit� de Grenoble I, 1997)

\bibitem[10]{Fan} Fan H J. Schroedinger equation, deformation theory and $tt^\ast$-geometry. arxiv:1107.1290

\bibitem[11]{FJR3}  Fan H J, Jarvis T J, Ruan Y B.
 The Witten equation, mirror symmetry, and quantum singularity theory.
Ann. Math. (2) {\bf 178}(1), 1-106  (2013)

\bibitem[12]{FJR2} Fan H J, Jarvis T J, Ruan Y B. The Witten equation and its virtual fundamental cycle. arXiv:0712.4025

\bibitem[13]{FJR1} Fan H J, Jarvis T J, Ruan Y B. A Mathematical Theory of the Gauged Linear Sigma Model.  arXiv:1506.02109




\bibitem[14]{Gaf} Gaffney P.  A special stokes's theorem for complete Riemannian manifolds. Ann. Math. (2) {\bf 60}, 140-145, (1954)



\bibitem[15]{GH} Griffiths P, Harris J. Principles in Algebraic Geometry. Pure and Applied Mathematics (Wiley-Interscience, New York, 1978)




\bibitem[16]{JKV} Jarvis T J, Kimura T, Vaintrob A. Moduli spaces of higher spin curves and integrable hierarchies. Compos. Math. {\bf 126}(2), 1571-212 (2001)



\bibitem[17]{LLS} Li C Z, Li S, Saito K. Primitive forms via polyvector fields. arXiv:1311.1659

\bibitem[18]{Liu} Liu K F. Holomorphic equivariant cohomology. Math. Ann. {\bf 303}(1), 125-148 (1995)

\bibitem[19]{MQ} Mathai V,  Quillen D. Superconnections, Thom classes, and equivariant differential forms.  Topology {\bf 25}(1),  85-110   (1986)


\bibitem[20]{Roy} Royden H, Fitzpatric P. Real analysis. 4 ed., New York, London, Macmillan (2010)




\bibitem[21]{Va}  Vafa C. Topological Landau-Ginzburg Models. Mod. Phys. Lett. A, {\bf 06}, 337 (1991)



\bibitem[22]{War} Warner H. Foundations of Differentiable manifolds and lie groups. Graduate Text in Mathematics. (2010)

\bibitem[23]{Wl}  Wlodarczyk J. Resolution of singularities of analytic spaces.  Proceedings of 15th G\"okova Geometry-Topology Conference, 31-63


\bibitem[24]{Wu}  Wu H. Bochner's skills in differential geometry (Part I). Advances in Mathematics(China). {\bf 10}(1), 57-76 (1981)
\end{thebibliography}

\end{document}